\documentclass[english]{article}
\usepackage[T1]{fontenc}
\usepackage[utf8]{inputenc}
\usepackage{mathrsfs}
\usepackage{url}
\usepackage{amsmath}
\usepackage{amsthm}
\usepackage{amssymb}

\usepackage{geometry}
\numberwithin{equation}{section}

\usepackage{fourier}
\usepackage{fancyhdr}
\newcommand{\R}{\mathbb{R}}
\newcommand{\p}{\ensuremath{\partial}}
\usepackage{graphics}
\usepackage{graphicx}\usepackage{pstricks}
\usepackage{pst-eucl}\usepackage{pst-plot}\usepackage{pst-tree}\usepackage{pstricks-add}\usepackage{tikz}

\makeatother

\theoremstyle{plain}
\newtheorem{thm}{\protect\theoremname}[section]
\theoremstyle{remark}
\newtheorem{notation}[thm]{\protect\notationname}
\theoremstyle{definition}
\newtheorem{defn}[thm]{\protect\definitionname}
\theoremstyle{plain}
\newtheorem{prop}[thm]{\protect\propositionname}
\theoremstyle{remark}
\newtheorem{rem}[thm]{\protect\remarkname}
\usepackage{babel}
\providecommand{\definitionname}{Definition}
\providecommand{\notationname}{Notation}
\providecommand{\propositionname}{Proposition}
\providecommand{\remarkname}{Remark}
\providecommand{\theoremname}{Theorem}

\date{}

\begin{document}
\title{Global existence and uniqueness for the planar four-velocity Broadwell
model with arbitrary $C^{1}$ bounded initial and boundary data in
a rectangular domain}
\author{Koudzo Togbévi Selom SOBAH\textsuperscript{\textsuperscript{1{*}}}\textsuperscript{}
and Amah Séna D'ALMEIDA\textsuperscript{\textsuperscript{2}}}

\maketitle
\textsuperscript{1,2}Department of Mathematics, Faculty of Sciences
and Laboratory of Mathematics and Applications, \\University of Lomé,
Lomé, TOGO

{*} corresponding email: deselium@gmail.com 
\begin{abstract}
This paper establishes the global existence of classical solutions
to the initial-boundary value problem for the planar four-velocity
Broadwell model in a multidimensional setting. While uniqueness was
already established in our previous work---which addressed existence
under the assumption of sufficiently small, bounded, and $C^{1}$-regular
initial and boundary data---the objective of the present study is
to completely lift this size restriction. By extending the previously
developed fixed-point framework, we prove that the methodology remains
highly effective for arbitrary initial and boundary data. Within this
generalized setting, both the solutions and their first-order partial
derivatives are shown to be well-defined and to exist globally for
all times.
\end{abstract}

\section*{Introduction}

The approximation of the Boltzmann equation through the discretization
of the velocity space leads to the framework of discrete kinetic equations.
By restricting particle velocities to a finite set of vectors, the
complex nonlinear integro-differential kinetic equation \cite{botlzman}
is transformed into a coupled system of semilinear hyperbolic partial
differential equations. Since the seminal contributions of Broadwell
\cite{3,4}, who formulated the first physically consistent discrete
velocity models for gas dynamics, a general theoretical framework
for binary \cite{1} and higher-order multiple collisions was established.
A central issue in discrete kinetic theory is the global existence
and uniqueness of classical solutions.The mathematical analysis in
one spatial dimension is now well-established, with comprehensive
global existence results for mixed problems \cite{5,6,11,14,compte rendu meca}.
For instance, fractional step techniques were successfully applied
to the two-velocity Carleman model \cite{2}, while exact solutions
were derived for the three-velocity Broadwell model \cite{7}. Furthermore,
the initial value problem for an arbitrary number of velocities was
solved in, and entropy-based formulations have led to bounded global
solutions \cite{Cabanne-kawashima} and classical solutions under
diverse boundary conditions \cite{kawasima}. In contrast, expanding
these results to multidimensional domains presents severe mathematical
obstacles. Early breakthroughs were limited to the stationary planar
four-velocity Broadwell model via fixed-point theorems \cite{Cercoi illner shinbrit},
while initial non-stationary attempts relied on fractional steps \cite{Toscani walus}.
Later investigations established existence and non-uniqueness features
for the two-dimensional stationary boundary value problem within the
general Broadwell framework $B_{\theta},\theta\in\left[0;\pi/2\right]$(\cite{defoou})
accompanied by exact analytical solutions \cite{dameida agosse,Nicou=00003D00003D0000E9}
and numerical treatments of mixed initial-boundary value problems
\cite{lakou agoss d almoe}. A broader 15-velocity model with three
moduli was also analyzed in \cite{d almeida}. Despite these contributions,
and apart from the early insights in \cite{Toscani walus}, the multidimensional
non-stationary problem has long remained a major open challenge. In
recent works \cite{arxiv 4V,arxiv Sob dalm 4Vgene}, we addressed
this gap by developing fixed-point methods to establish the global
existence of solutions for the multidimensional initial-boundary value
problem. However, that analysis was subject to the assumption of sufficiently
small data with their first order partial derivatives.\\
  The objective of the present paper is to completely lift this restriction.
We demonstrate that our previously established fixed-point framework
remains robust enough to overcome the smallness constraint. More precisely,
we prove the global existence and uniqueness of solutions to the multidimensional
initial-boundary value problem for the standard planar four-velocity
Broadwell under arbitrary (large) initial data with their first order
partial derivatives. Under this framework, the solutions and their
derivatives are well-defined and exist for all times; however, the
asymptotic behavior of these derivatives as $t\to+\infty$ (specifically,
whether they remain uniformly bounded or grow indefinitely) remains
an open question. \\
The remainder of this paper is organized as follows. In Section \ref{sec:Preliminary},
we give some preliminaries and state the main results of the paper;
in Section \ref{sec:Positiveness-of-the}, we prove the positivity
of the solutions for positive data and in Section \ref{sec:Existence},
we prove the existence of solutions.

\section{Preliminaries }\label{sec:Preliminary}

We extend the framework presented in the Section 1 of \cite{arxiv Sob dalm 4Vgene}
where $K$ be a compact convex non-empty subset of $\R^{d},$$d$
integer, $d\geq1$ and let $\mathcal{H}\equiv\{H_{m}\}^{m_{0}}_{m=1}$
be a finite family of pair-wise distinct affine hyper-planes of $\R^{d}$.
Recall the notations 
\begin{notation}
\label{wqjjqu} For $f$ bounded real valued function with domain
$D_{f}$, $\|f\|_{\infty}\equiv{\displaystyle \sup_{x\in D_{f}}}|f(x)|$.
When $D_{f}\subset\mathbb{R}^{d}$ and $f\equiv f(x_{1},\dots,x_{d})$
and ${\displaystyle \frac{\partial f}{\partial x_{\alpha}}}$ (for
$\alpha=1,\dots,d$) are bounded, $\|f\|_{1}\equiv\max\left\{ \|f\|_{\infty},\,{\displaystyle \max_{1\le\alpha\le d}}\left\Vert \frac{\partial f}{\partial x_{\alpha}}\right\Vert _{\infty}\right\} .$
\end{notation}

\begin{notation}
\label{iqiqjja}For $p\ge2$ and any function $F=(f_{1},\dots,f_{p})$
with $f_{i}$ real-valued and bounded, 
\begin{equation}
\|F\|\doteq{\displaystyle \max_{1\le i\le p}}\|f_{i}\|_{\infty}.\label{oosaz}
\end{equation}
\end{notation}

\begin{notation}
\label{kqqkkd} Let $C\left(K,\R\right)$, the space of continuous
functions from $K$ on $\R,$ be endowed with $\left\Vert \cdot\right\Vert _{\infty}$.
Let $\mathscr{E}_{\mathcal{H}}$ be the subspace of $C\left(K,\R\right)$
consisting of functions $f$ such that $\forall\alpha=1,\cdots,d,$$\dfrac{\partial f}{\partial x_{\alpha}}$
is defined everywhere on $\operatorname{int}(K)\setminus{\displaystyle \cup^{m_{0}}_{m=1}}H_{m}$
and is continuous and bounded.
\end{notation}

\begin{defn}
\label{nnjnjn} For $R,R'>0,$ let us define $\mathscr{M}_{R,R'}\equiv\left\{ \right.F\in\left(\mathscr{E}_{\mathcal{H}}\right)^{p}/\left\Vert F\right\Vert <R\text{ and }{\displaystyle \max_{1\le\alpha\le d}}\left\Vert \dfrac{\partial F}{\partial x_{\alpha}}\right\Vert <R'\left.\right\} $
and $\mathscr{M}^{+}_{R,R'}\equiv\left\{ \right.F\in\mathscr{M}_{R,R'}/F\geq0\left.\right\} $
. 
\end{defn}

Then we prove exactly as in the Proposition 1.8 in \cite{arxiv Sob dalm 4Vgene}
that 
\begin{prop}
\label{prop::odlp-1} $\mathscr{M}_{R,R'}$ and consequently \textup{$\mathscr{M}^{+}_{R,R'}$
are} non-empty convex subsets of \textup{$C\left(K,\R\right)^{p}$
}and are relatively compact in $\left(C\left(K,\R\right)^{p},\left\Vert \cdot\right\Vert \right).$ 
\end{prop}

Now we recall the kinetic equations for the four-velocity Broadwell
model in the plane, where $N_{i}\equiv N_{i}(t,x,y)\in\R:$ 
\begin{equation}
\begin{cases}
{\textstyle \dfrac{\p N_{1}}{\p t}+c\dfrac{\p N_{1}}{\p x}}=Q\left(N\right)\\
\\\dfrac{\p N_{2}}{\p t}+c\dfrac{\p N_{2}}{\p y}=-Q\left(N\right)\\
\\\dfrac{\p N_{3}}{\p t}-c\dfrac{\p N_{3}}{\p y}=-Q\left(N\right)\\
\\\dfrac{\p N_{4}}{\p t}-c\dfrac{\p N_{4}}{\p x}=Q\left(N\right)
\end{cases}\label{eq:koiqmn}
\end{equation}

\begin{equation}
Q\left(N\right)=2cS\left(N_{2}N_{3}-N_{1}N_{4}\right),\label{eq:jjdpoa}
\end{equation}
$c,S>0$ constants. 

We recall the problems $\Sigma$ and $\Sigma_{\tau,\tau'}$ posed
in \cite{arxiv 4V}. $\left[a_{1},b_{1}\right]\times\left[a_{2},b_{2}\right]\subset\R^{2}.$
$\Sigma$ is system \ref{eq:koiqmn} on $\left[0;+\infty\right[\times\left[a_{1},b_{1}\right]\times\left[a_{2},b_{2}\right]$
with the initial and boundary conditions \ref{eq:lsos}-\ref{eq:ikioi-1}

\begin{align}
N_{i}\left(0,x,y\right)= & N^{0}_{i}\left(x,y\right),\;\left(x,y\right)\in\left[a_{1};b_{1}\right]\times\left[a_{2};b_{2}\right],i=1,\cdots,4\label{eq:lsos}\\
N_{1}\left(t,a_{1},y\right)= & N^{-}_{1}\left(t,y\right),\;\left(t,y\right)\in\left[0;+\infty\right[\times\left[a_{2};b_{2}\right]\label{eq:losso}\\
N_{2}\left(t,x,a_{2}\right)= & N^{-}_{2}\left(t,x\right),\;\left(t,x\right)\in\left[0;+\infty\right[\times\left[a_{1};b_{1}\right]\label{eq:lsoo}\\
N_{3}\left(t,x,b_{2}\right)= & N^{+}_{3}\left(t,x\right),\;\left(t,x\right)\in\left[0;+\infty\right[\times\left[a_{1};b_{1}\right]\label{eq:ikioi}\\
N_{4}\left(t,b_{1},y\right)= & N^{+}_{4}\left(t,y\right),\;\left(t,y\right)\in\left[0;+\infty\right[\times\left[a_{2};b_{2}\right]\label{eq:ikioi-1}
\end{align}
The initial and boundary data satisfy the compatibility conditions
\ref{eq:kqiiq}-\ref{eq:looqp-1} 

\begin{align}
N^{0}_{1}\left(a_{1},y\right) & =N^{-}_{1}\left(0,y\right),\;y\in\left[a_{2};b_{2}\right]\label{eq:kqiiq}\\
N^{0}_{2}\left(x,a_{2}\right) & =N^{-}_{2}\left(0,x\right),\;x\in\left[a_{1};b_{1}\right]\label{eq:lqoop}\\
N^{0}_{3}\left(x,b_{2}\right) & =N^{+}_{3}\left(0,x\right),\;x\in\left[a_{1};b_{1}\right]\label{eq:loppq-1}\\
N^{0}_{4}\left(b_{1},y\right) & =N^{+}_{4}\left(0,y\right),\;y\in\left[a_{2};b_{2}\right]\label{eq:looqp-1}
\end{align}

For $\left[\tau,\tau'\right]\subset\R^{+},$ $\Sigma_{\tau,\tau'}$
is \ref{eq:koiqmn} on $\left[\tau,\tau'\right]\times\left[a_{1},b_{1}\right]\times\left[a_{2},b_{2}\right]$
with the initial and boundary conditions \ref{eq:lsos-1}-\ref{eq:ikioi-1-1}
\begin{align}
N_{i}\left(\tau,x,y\right)= & N^{\tau}_{i}\left(x,y\right),\;\left(x,y\right)\in\left[a_{1};b_{1}\right]\times\left[a_{2};b_{2}\right],i=1,\cdots,4\label{eq:lsos-1}\\
N_{1}\left(t,a_{1},y\right)= & N^{-}_{1}\left(t,y\right),\;\left(t,y\right)\in\left[\tau,\tau'\right]\times\left[a_{2};b_{2}\right]\label{eq:losso-1}\\
N_{2}\left(t,x,a_{2}\right)= & N^{-}_{2}\left(t,x\right),\;\left(t,x\right)\in\left[\tau,\tau'\right]\times\left[a_{1};b_{1}\right]\label{eq:lsoo-1}\\
N_{3}\left(t,x,b_{2}\right)= & N^{+}_{3}\left(t,x\right),\;\left(t,x\right)\in\left[\tau,\tau'\right]\times\left[a_{1};b_{1}\right]\label{eq:ikioi-2}\\
N_{4}\left(t,b_{1},y\right)= & N^{+}_{4}\left(t,y\right),\;\left(t,y\right)\in\left[\tau,\tau'\right]\times\left[a_{2};b_{2}\right]\label{eq:ikioi-1-1}
\end{align}
and the compatibility conditions \ref{eq:kqiiq-1}-\ref{eq:looqp-1-1}
\begin{align}
N^{\tau}_{1}\left(a_{1},y\right) & =N^{-}_{1}\left(\tau,y\right),\;y\in\left[a_{2};b_{2}\right]\label{eq:kqiiq-1}\\
N^{\tau}_{2}\left(x,a_{2}\right) & =N^{-}_{2}\left(\tau,x\right),\;x\in\left[a_{1};b_{1}\right]\label{eq:lqoop-1}\\
N^{\tau}_{3}\left(x,b_{2}\right) & =N^{+}_{3}\left(\tau,x\right),\;x\in\left[a_{1};b_{1}\right]\label{eq:loppq-1-1}\\
N^{\tau}_{4}\left(b_{1},y\right) & =N^{+}_{4}\left(\tau,y\right),\;y\in\left[a_{2};b_{2}\right]\label{eq:looqp-1-1}
\end{align}

The data and their first order partial derivatives are assumed to
be continuous and bounded. More over the data are assumed to be non-negative. 

We shall use $\left\Vert u\right\Vert _{\infty}={\displaystyle \sup_{\phi\in D_{u}}}\left|u\left(\phi\right)\right|$
for bounded real functions $u\equiv u\left(\phi\right)$ with domain
$D_{u}$; for bounded real functions $g\equiv g\left(t,x,y\right)$
with bounded partial derivatives, we shall use $\left\Vert g\right\Vert _{1}\equiv\max\left\{ \left\Vert g\right\Vert _{\infty},\left\Vert \dfrac{\partial g}{\partial t}\right\Vert _{\infty},\left\Vert \dfrac{\partial g}{\partial x}\right\Vert _{\infty},\left\Vert \dfrac{\partial g}{\partial y}\right\Vert _{\infty}\right\} .$ 

Let us set for $\left(t,x,y\right)\in\left[\tau,\tau'\right]\times\left[a_{1},b_{1}\right]\times\left[a_{2},b_{2}\right],$
\begin{align}
\overline{N^{\tau}_{1}}\left(t,x,y\right) & \equiv N^{\tau}_{1}\left(x-c\left(t-\tau\right),y\right)\label{eq:kqiiq-1-1}\\
\overline{N^{\tau}_{2}}\left(t,x,y\right) & \equiv N^{\tau}_{2}\left(x,y-c\left(t-\tau\right)\right)\label{eq:lqoop-1-1}\\
\overline{N^{\tau}_{3}}\left(t,x,y\right) & \equiv N^{\tau}_{3}\left(x,y+c\left(t-\tau\right)\right)\label{eq:loppq-1-1-1}\\
\overline{N^{\tau}_{4}}\left(t,x,y\right) & \equiv N^{\tau}_{4}\left(x-c\left(t-\tau\right),y\right)\label{eq:looqp-1-1-1}
\end{align}
and for $\left(t,x,y\right)\in\left[0;+\infty\right[\times\left[a_{1},b_{1}\right]\times\left[a_{2},b_{2}\right],$
when they are defined 
\begin{align}
\overline{N^{-}_{1}}\left(t,x,y\right) & \equiv N^{-}_{1}\left(t-\frac{x-a_{1}}{c},y\right)\label{eq:kqiiq-1-1-1}\\
\overline{N^{-}_{2}}\left(t,x,y\right) & \equiv N^{-}_{2}\left(t-\frac{y-a_{2}}{c},x\right)\label{eq:lqoop-1-1-1}\\
\overline{N^{+}_{3}}\left(t,x,y\right) & \equiv N^{+}_{3}\left(t-\frac{b_{2}-y}{c},x\right)\label{eq:loppq-1-1-1-1}\\
\overline{N^{+}_{4}}\left(t,x,y\right) & \equiv N^{+}_{4}\left(t-\frac{b_{1}-x}{c},y\right).\label{eq:looqp-1-1-1-1}
\end{align}

Let us denote $\mathscr{P}\equiv\left[\tau,\tau'\right]\times\left[a_{1},b_{1}\right]\times\left[a_{2},b_{2}\right].$
In \cite{arxiv 4V} we introduce the continuous operator $\mathcal{T}:C\left(\mathscr{P},\R\right)^{4}\longrightarrow C\left(\mathscr{P},\R\right)^{4},$
with the following property:
\begin{prop}
\label{prop:Let-the-4-tuple} $N=\left(N_{i}\right)^{4}_{i=1}$ continuous
on $\mathscr{P}$ and possessing all first-order partial derivatives
is a solution of problem $\Sigma_{\tau,\tau'}$ iff $N$ is a fixed
point of $\mathcal{T}.$ 
\end{prop}

For an inequality $f\left(t,x,y\right)\leq0$ defined on $\mathscr{P}$,
let $\mathbb{I}_{f\left(t,x,y\right)\leq0}$ be the identity function
of the set \\$\left\{ \left(t,x,y\right)\in\mathscr{P}:f\left(t,x,y\right)\leq0\right\} .$
Then $\mathcal{T}$ is defined by 
\begin{equation}
\mathcal{T}_{1}\left(M\right)\left(t,x,y\right)=\mathcal{T}^{A}_{1}\left(M\right)\left(t,x,y\right)\cdot\mathbb{I}_{x-c\left(t-\tau\right)\geq a_{1}}\left(t,x,y\right)+
\mathcal{T}^{B}_{1}\left(M\right)\left(t,x,y\right)\cdot\mathbb{I}_{x-c\left(t-\tau\right)\leq a_{1}}\left(t,x,y\right)\label{aoalal-2}
\end{equation}
where

\begin{align}
\mathcal{T}^{A}_{1}\left(M\right)\left(t,x,y\right) & ={\displaystyle \int^{t}_{\tau}}Q\left(M\right)\left(s,x+c\left(s-t\right),y\right)ds+\overline{N^{\tau}_{1}}\left(t,x,y\right)\label{eq:olole-1-3}
\end{align}

\begin{align}
\mathcal{T}^{B}_{1}\left(M\right)\left(t,x,y\right) & ={\displaystyle \int^{t}_{t-\frac{1}{c}x+\frac{a_{1}}{c}}}Q\left(M\right)\left(s,x+c\left(s-t\right),y\right)ds+\overline{N^{-}_{1}}\left(t,x,y\right)\label{eq:olole-1-1-2}
\end{align}
\begin{equation}
\mathcal{T}_{2}\left(M\right)\left(t,x,y\right)=\mathcal{T}^{A}_{2}\left(M\right)\left(t,x,y\right)\cdot\mathbb{I}_{y-c\left(t-\tau\right)\geq a_{2}}\left(t,x,y\right)+
\mathcal{T}^{B}_{2}\left(M\right)\left(t,x,y\right)\cdot\mathbb{I}_{y-c\left(t-\tau\right)\leq a_{2}}\left(t,x,y\right)\label{aoalal-1-2}
\end{equation}

where 
\begin{align}
\mathcal{T}^{A}_{2}\left(M\right)\left(t,x,y\right) & ={\textstyle {\displaystyle \int^{t}_{\tau}}-Q\left(M\right)}\left(s,x,y+c\left(s-t\right)\right)ds+{\textstyle \overline{N^{\tau}_{2}}\left(t,x,y\right)}\label{eq:olole-1-1-1-1-1-1-1}
\end{align}

\begin{align}
\mathcal{T}^{B}_{2}\left(M\right)\left(t,x,y\right) & ={\textstyle {\displaystyle \int^{t}_{t-\frac{1}{c}y+\frac{a_{2}}{c}}}}-Q\left(M\right)\left(s,x,y+c\left(s-t\right)\right)ds+\overline{N^{-}_{2}}\left(t,x,y\right)\label{eq:olole-1-1-1-1-2-1}
\end{align}

\begin{equation}
\mathcal{T}_{3}\left(M\right)\left(t,x,y\right)=\mathcal{T}^{A}_{3}\left(M\right)\left(t,x,y\right)\cdot\mathbb{I}_{y+c\left(t-\tau\right)\leq b_{2}}\left(t,x,y\right)+
\mathcal{T}^{B}_{3}\left(M\right)\left(t,x,y\right)\cdot\mathbb{I}_{y+c\left(t-\tau\right)\geq b_{2}}\left(t,x,y\right)+\label{aoalal-1-1-3}
\end{equation}
where 
\begin{align}
\mathcal{T}^{A}_{3}\left(M\right)\left(t,x,y\right) & ={\displaystyle \int^{t}_{\tau}}-Q\left(M\right)\left(s,x,y-c\left(s-t\right)\right)ds+\overline{N^{\tau}_{3}}\left(t,x,y\right)\label{eq:olole-1-2-2}
\end{align}
\begin{align}
\mathcal{T}^{B}_{3}\left(M\right)\left(t,x,y\right) & ={\displaystyle \int^{t}_{t+\frac{1}{c}y-\frac{b_{2}}{c}}}-Q\left(M\right)\left(s,x,y-c\left(s-t\right)\right)ds+\overline{N^{+}_{3}}\left(t,x,y\right)\label{eq:olole-1-1-1-1}
\end{align}

\begin{equation}
\mathcal{T}_{4}\left(M\right)\left(t,x,y\right)=\mathcal{T}^{A}_{4}\left(M\right)\left(t,x,y\right)\cdot\mathbb{I}_{x+c\left(t-\tau\right)\leq b_{1}}\left(t,x,y\right)+
\mathcal{T}^{B}_{4}\left(M\right)\left(t,x,y\right)\cdot\mathbb{I}_{x+c\left(t-\tau\right)\geq b_{1}}\left(t,x,y\right)\label{aoalal-1-1-1-1}
\end{equation}
where 
\begin{align}
\mathcal{T}^{A}_{4}\left(M\right)\left(t,x,y\right) & ={\displaystyle \int^{t}_{\tau}}Q\left(M\right)\left(s,x-c\left(s-t\right),y\right)ds+\overline{N^{\tau}_{4}}\left(t,x,y\right)\label{eq:olole-1-2-1-2}
\end{align}
\begin{align}
\mathcal{T}^{B}_{4}\left(M\right)\left(t,x,y\right) & =\int^{t}_{t+\frac{1}{c}x-\frac{b_{1}}{c}}Q\left(M\right)\left(s,x-c\left(s-t\right),y\right)ds+\overline{N^{+}_{4}}\left(t,x,y\right).\label{eq:olole-1-1-1-2-2}
\end{align}
In \cite{arxiv 4V}, we also introduce another continuous operator
$\mathcal{T}^{\sigma}:C\left(\mathscr{P},\R\right)^{4}\longrightarrow C\left(\mathscr{P},\R\right)^{4}$
which satisfies the following property:
\begin{prop}
\label{prop:pqmm} $\mathcal{T}^{\sigma}$ is non-negative when $\sigma>0$
is sufficiently large. Suppose that the 4-tuple $N=\left(N_{i}\right)^{4}_{i=1}$
is continuous on $\mathscr{P}$ and possesses all first-order partial
derivatives. If $N$ is a fixed point of $\mathcal{T}^{\sigma}$ then
$N$ is a non-negative solution of problem $\Sigma_{\tau,\tau'}.$
\end{prop}

Let us recall the explicit expression of the operator $\mathcal{T}^{\sigma}:$
first put 
\begin{equation}
\rho\left(N\right)=\sum^{4}_{i=1}N_{i}\label{eq:kiid}
\end{equation}
and 

\begin{equation}
\begin{cases}
Q^{\sigma}_{1}\left(N\right)=\sigma\rho\left(N\right)N_{1}+Q\left(N\right)\\
Q^{\sigma}_{2}\left(N\right)=\sigma\rho\left(N\right)N_{2}-Q\left(N\right)\\
Q^{\sigma}_{3}\left(N\right)=\sigma\rho\left(N\right)N_{3}-Q\left(N\right)\\
Q^{\sigma}_{4}\left(N\right)=\sigma\rho\left(N\right)N_{4}+Q\left(N\right)
\end{cases}.\label{eq:opaole}
\end{equation}
Put $\left|N\right|=\left(\left|N_{i}\right|\right)^{4}_{i=1}$. We
have $\mathcal{T}^{\sigma}\left(M\right)=\left(\mathcal{T}^{\sigma}_{i}\left(M\right)\right)^{4}_{i=1}$
where 
\begin{equation}
\mathcal{T}^{\sigma}_{1}\left(M\right)\left(t,x,y\right)=\mathcal{T}^{A,\sigma}_{1}\left(M\right)\left(t,x,y\right)\cdot\mathbb{I}_{x-c\left(t-\tau\right)\geq a_{1}}\left(t,x,y\right)+
\mathcal{T}^{B,\sigma}_{1}\left(M\right)\left(t,x,y\right)\cdot\mathbb{I}_{x-c\left(t-\tau\right)\leq a_{1}}\left(t,x,y\right)\label{aoalal}
\end{equation}
with

\begin{multline}
\mathcal{T}^{A,\sigma}_{1}\left(M\right)\left(t,x,y\right)=\Biggl({\displaystyle \int^{t}_{\tau}}e^{\sigma\int^{s}_{\tau}\rho\left(\left|M\right|\right)\left(r,x+c\left(r-t\right),y\right)dr}\cdot Q^{\sigma}_{1}\left(\left|M\right|\right)
\left(s,x+c\left(s-t\right),y\right)ds+\overline{N^{\tau}_{1}}\left(t,x,y\right)\Biggr)\cdot\\
e^{-\sigma{\textstyle {\displaystyle \int^{t}_{\tau}}}\rho\left(\left|M\right|\right)\left(s,x+c\left(s-t\right),y\right)ds}\label{eq:olole-1}
\end{multline}

\begin{multline}
\mathcal{T}^{B,\sigma}_{1}\left(M\right)\left(t,x,y\right)=
\Biggl({\displaystyle \int^{t}_{t-\frac{1}{c}x+\frac{a_{1}}{c}}}e^{\sigma\int^{s}_{t-\frac{1}{c}x+\frac{a_{1}}{c}}\rho\left(\left|M\right|\right)\left(r,x+c\left(r-t\right),y\right)dr}\cdot Q^{\sigma}_{1}\left(\left|M\right|\right)
\left(s,x+c\left(s-t\right),y\right)ds+\overline{N^{-}_{1}}\left(t,x,y\right)\Biggr)\cdot\\
e^{-\sigma{\textstyle \int^{t}_{t-\frac{1}{c}x+\frac{a_{1}}{c}}}\rho\left(\left|M\right|\right)\left(s,x+c\left(s-t\right),y\right)ds}\label{eq:olole-1-1}
\end{multline}
\begin{equation}
\mathcal{T}^{\sigma}_{2}\left(M\right)\left(t,x,y\right)=\mathcal{T}^{A,\sigma}_{2}\left(M\right)\left(t,x,y\right)\cdot\mathbb{I}_{y-c\left(t-\tau\right)\geq a_{2}}\left(t,x,y\right)+
\mathcal{T}^{B,\sigma}_{2}\left(M\right)\left(t,x,y\right)\cdot\mathbb{I}_{y-c\left(t-\tau\right)\leq a_{2}}\left(t,x,y\right)\label{aoalal-1}
\end{equation}

with
\begin{multline}
\mathcal{T}^{A,\sigma}_{2}\left(M\right)\left(t,x,y\right)=\Biggl({\textstyle {\displaystyle \int^{t}_{\tau}}e^{\sigma\int^{s}_{\tau}\rho\left(\left|M\right|\right)\left(r,x,y+c\left(r-t\right)\right)dr}}\cdot
{\textstyle Q^{\sigma}_{2}\left(\left|M\right|\right)\left(s,x,y+c\left(s-t\right)\right)ds}+{\textstyle \overline{N^{\tau}_{2}}\left(t,x,y\right)\Biggr)}\cdot\\
{\textstyle e^{-\sigma{\displaystyle \int^{t}_{\tau}}\rho\left(\left|M\right|\right)\left(s,x,y+c\left(s-t\right)\right)ds}}\label{eq:olole-1-1-1-1-1-1}
\end{multline}

\begin{multline}
\mathcal{T}^{B,\sigma}_{2}\left(M\right)\left(t,x,y\right)=
\Biggl({\textstyle {\textstyle {\displaystyle \int^{t}_{t-\frac{1}{c}y+\frac{a_{2}}{c}}}}e^{\sigma\int^{s}_{t-\frac{1}{c}y+\frac{a_{2}}{c}}\rho\left(\left|M\right|\right)\left(r,x,y+c\left(r-t\right)\right)dr}}\cdot
{\textstyle Q^{\sigma}_{2}\left(\left|M\right|\right)\left(s,x,y+c\left(s-t\right)\right)ds}+{\textstyle \overline{N^{-}_{2}}\left(t,x,y\right)\Biggr)}\cdot\\
{\textstyle e^{-\sigma{\displaystyle \int^{t}_{t-\frac{1}{c}y+\frac{a_{2}}{c}}}\rho\left(\left|M\right|\right)\left(s,x,y+c\left(s-t\right)\right)ds}}\label{eq:olole-1-1-1-1-2}
\end{multline}

\begin{equation}
\mathcal{T}^{\sigma}_{3}\left(M\right)\left(t,x,y\right)=\mathcal{T}^{A,\sigma}_{3}\left(M\right)\left(t,x,y\right)\cdot\mathbb{I}_{y+c\left(t-\tau\right)\leq b_{2}}\left(t,x,y\right)+
\mathcal{T}^{B,\sigma}_{3}\left(M\right)\left(t,x,y\right)\cdot\mathbb{I}_{y+c\left(t-\tau\right)\geq b_{2}}\left(t,x,y\right)\label{aoalal-1-1}
\end{equation}
with 
\begin{multline}
\mathcal{T}^{A,\sigma}_{3}\left(M\right)\left(t,x,y\right)=
\Biggl({\displaystyle \int^{t}_{\tau}}e^{\sigma\int^{s}_{0}\rho\left(\left|M\right|\right)\left(r,x,y-c\left(r-t\right)\right)dr}\cdot Q^{\sigma}_{3}\left(\left|M\right|\right)
\left(s,x,y-c\left(s-t\right)\right)ds+\overline{N^{\tau}_{3}}\left(t,x,y\right)\Biggr)\cdot\\
e^{-\sigma{\textstyle {\displaystyle \int^{t}_{\tau}}}\rho\left(\left|M\right|\right)\left(s,x,y-c\left(s-t\right)\right)ds}\label{eq:olole-1-2}
\end{multline}
\begin{multline}
\mathcal{T}^{B,\sigma}_{3}\left(M\right)\left(t,x,y\right)=
\Biggl({\displaystyle \int^{t}_{t+\frac{1}{c}y-\frac{b_{2}}{c}}}e^{\sigma\int^{s}_{t+\frac{1}{c}y-\frac{b_{2}}{c}}\rho\left(\left|M\right|\right)\left(r,x,y-c\left(r-t\right)\right)dr}\cdot Q^{\sigma}_{3}\left(\left|M\right|\right)
\left(s,x,y-c\left(s-t\right)\right)ds+\overline{N^{+}_{3}}\left(t,x,y\right)\Biggr)\cdot\\
e^{-\sigma{\textstyle {\displaystyle \int^{t}_{t+\frac{1}{c}y-\frac{b_{2}}{c}}}}\rho\left(\left|M\right|\right)\left(s,x,y-c\left(s-t\right)\right)ds}\label{eq:olole-1-1-1}
\end{multline}

\begin{equation}
\mathcal{T}^{\sigma}_{4}\left(M\right)\left(t,x,y\right)=\mathcal{T}^{A,\sigma}_{4}\left(M\right)\left(t,x,y\right)\cdot\mathbb{I}_{x+c\left(t-\tau\right)\leq b_{1}}\left(t,x,y\right)+
\mathcal{T}^{B,\sigma}_{4}\left(M\right)\left(t,x,y\right)\cdot\mathbb{I}_{x+c\left(t-\tau\right)\geq b_{1}}\left(t,x,y\right)\label{aoalal-1-1-1}
\end{equation}
with 
\begin{multline}
\mathcal{T}^{A,\sigma}_{4}\left(M\right)\left(t,x,y\right)=
\Biggl({\displaystyle \int^{t}_{\tau}}e^{\sigma\int^{s}_{\tau}\rho\left(\left|M\right|\right)\left(r,x-c\left(r-t\right),y\right)dr}\cdot Q^{\sigma}_{4}\left(\left|M\right|\right)
\left(s,x-c\left(s-t\right),y\right)ds+\overline{N^{\tau}_{4}}\left(t,x,y\right)\Biggr)\cdot\\
e^{-\sigma{\textstyle {\displaystyle \int^{t}_{\tau}}}\rho\left(\left|M\right|\right)\left(s,x-c\left(s-t\right),y\right)ds}\label{eq:olole-1-2-1}
\end{multline}
\begin{multline}
\mathcal{T}^{B,\sigma}_{4}\left(M\right)\left(t,x,y\right)=
\Biggl({\displaystyle \int^{t}_{t+\frac{1}{c}x-\frac{b_{1}}{c}}}e^{\sigma\int^{s}_{t+\frac{1}{c}x-\frac{b_{1}}{c}}\rho\left(\left|M\right|\right)\left(r,x-c\left(r-t\right),y\right)dr}\cdot Q^{\sigma}_{4}\left(\left|M\right|\right)
\left(s,x-c\left(s-t\right),y\right)ds+\overline{N^{+}_{4}}\left(t,x,y\right)\Biggr)\cdot\\
e^{-\sigma{\textstyle {\displaystyle \int^{t}_{t+\frac{1}{c}x-\frac{b_{1}}{c}}}}\rho\left(\left|M\right|\right)\left(s,x-c\left(s-t\right),y\right)ds}\label{eq:olole-1-1-1-2}
\end{multline}

The uniqueness of the solution to problems $\Sigma_{\tau,\tau'}$
and $\Sigma$ is established in \cite{arxiv 4V} for any value of
the data. This is achieved by using the operator $\mathcal{T}$. 

In \cite{arxiv 4V} the data are assumed to be non-negative; and we
have proved the existence of a non-negative global bounded solution
with its derivatives, for data which are small enough with their derivatives.
We have particularly established that the uniform bounds of the solution
and its derivatives do not exceed the ones of the data. 

In the present paper, we complete the study undertaken in \cite{arxiv 4V},
by addressing the more general case where the data and their derivatives
can possess any arbitrary bound. First, we use the operator $\mathcal{T}$
to prove that the solution to the problem $\Sigma$ is positive when
the data is positive. 

Then we mainly prove that the solution to the problem $\Sigma$ exists
and is non-negative for any non-negative value of the data. Particularly
we established that the solution is bounded and global. The unknown
feature that remains to be established is whether the derivatives
of the solution are also bounded for all time. Recall that in \cite{arxiv 4V},
we have proven that the derivatives of the solution are bounded in
time, when the data and their derivatives are sufficiently small.

\begin{notation}
Put 
\begin{equation}
q\equiv\max_{1\leq i\leq4}\Biggl\{\left\Vert \overline{N^{\tau}_{i}}\right\Vert _{\infty},\left\Vert \overline{N^{-}_{1}}\right\Vert _{\infty},\left\Vert \overline{N^{-}_{2}}\right\Vert _{\infty},\left\Vert \overline{N^{+}_{3}}\right\Vert _{\infty},\left\Vert \overline{N^{+}_{4}}\right\Vert _{\infty}\Biggr\}\label{bghghf}
\end{equation}
and 
\begin{equation}
q'\equiv\max_{1\leq i\leq4}\left\{ \left\Vert \nabla_{t,x,y}\overline{N^{\tau}_{i}}\right\Vert ,\left\Vert \nabla_{t,x,y}\overline{N^{-}_{1}}\right\Vert ,\left\Vert \nabla_{t,x,y}\overline{N^{-}_{2}}\right\Vert ,\left\Vert \nabla_{t,x,y}\overline{N^{+}_{3}}\right\Vert ,\left\Vert \nabla_{t,x,y}\overline{N^{+}_{4}}\right\Vert \right\} .\label{oloo-1}
\end{equation}

Let $R_{0}>0$ be given. Put 
\begin{equation}
f\left(q\right)\equiv\dfrac{1}{4\sigma\left(1+c\right)\left(4\sigma+4cS\right)R^{2}_{0}+\left(1+c\right)\left(8\sigma+8cS\right)R_{0}+16\left(\sigma+cS\right)\left(1+c\right)q}\label{ldod}
\end{equation}
\begin{equation}
g\left(q\right)\equiv\dfrac{R_{0}-q}{\left(8\sigma+8cS\right)R^{2}_{0}}\label{loodie}
\end{equation}
\end{notation}

In the sequel, we prove the following results
\begin{thm}
\label{thm:If-the-data}If the data are positive, so is the solution
to the problems $\Sigma$ and $\Sigma_{\tau,\tau'}$.
\end{thm}

\begin{thm}
\label{thm:dhgrt}Let $q<R_{0}$ and $\tau'-\tau\leq\min\left\{ 1,f\left(q\right),g\left(q\right)\right\} .$
Then the problem $\Sigma_{\tau,\tau'}$ admits an unique continuous
solution $N$ which is non-negative with bounded first-order partial
derivatives. Moreover $N$ satisfies 
\begin{equation}
\left\Vert N\right\Vert <R_{0}\label{iido}
\end{equation}
and \textup{
\begin{equation}
\max\left\{ \left\Vert \dfrac{\partial N}{\partial t}\right\Vert ,\left\Vert \dfrac{\partial N}{\partial x}\right\Vert ,\left\Vert \dfrac{\partial N}{\partial y}\right\Vert \right\} <
\dfrac{\left(8+\frac{4}{c}\right)\left(\sigma+cS\right)R^{2}_{0}+4\sigma\left(4\sigma+4cS\right)R^{3}_{0}+\left(8+\frac{4}{c}\right)\sigma R_{0}q+q'}{1-\left(\tau'-\tau\right)\left[4\sigma\left(1+c\right)\left(4\sigma+4cS\right)R^{2}_{0}+\left(1+c\right)\left(8\sigma+8cS\right)R_{0}+4\sigma\left(1+c\right)q\right]}.
\end{equation}
}
\end{thm}

\begin{thm}
\label{thm:The-problem-} The problem $\Sigma$ admits an unique non-negative
continuous and bounded solution $N.$ More ever $N$ satisfies 
\begin{equation}
\left\Vert N\right\Vert \leq{\displaystyle \max_{1\leq i\leq4}}\left\{ \right.\left\Vert N^{0}_{i}\right\Vert _{\infty},\left\Vert N^{-}_{1}\right\Vert _{\infty},\left\Vert N^{-}_{2}\right\Vert _{\infty},\left\Vert N^{+}_{3}\right\Vert _{\infty},\left\Vert N^{+}_{4}\right\Vert _{\infty}\left.\right\} .
\end{equation}
\end{thm}

\section{Positivity of the solution}\label{sec:Positiveness-of-the}

\begin{proof}[\textbf{Proof of theorem \ref{thm:If-the-data}}]
Let $M$ be a continuous maximal solution to the problem $\Sigma.$
The restriction of $M$ to $\left[\tau,\tau'\right]\times\left[a_{1},b_{1}\right]\times\left[a_{2},b_{2}\right]$
is denoted by $M^{\tau,\tau'},$ for any $0\leq\tau<\tau'.$ we also
denote by $\mathcal{T}^{\tau,\tau'}$ the operator $\mathcal{T}$
in Proposition \eqref{prop:Let-the-4-tuple}. Let $T>0$ such that
$M$ is defined on $\left[0;T\right]$. Put $R=\left\Vert M^{0,T}\right\Vert .$
$R>0.$

Let $0<\tau_{1}\leq T.$ $M^{0,\tau_{1}}$ is solution to the problem
$\Sigma_{0,\tau_{1}}$ where the initial data are $N^{0}_{i}.$ For
$\left(t,x,y\right)\in\left[0,\tau_{1}\right]\times\left[a_{1},b_{1}\right]\times\left[a_{2},b_{2}\right]$,
we have (\eqref{eq:olole-1-3})
\begin{align}
\mathcal{T}^{0,\tau_{1},A}_{1}\left(M^{0,\tau_{1}}\right)\left(t,x,y\right) & =tQ\left(M^{0,\tau_{1}}\right)\left(s_{0},x+c\left(s_{0}-t\right),y\right)+\overline{N^{0}_{1}}\left(t,x,y\right)\label{eq:olole-1-3-1}
\end{align}
where $s_{0}\in\left[0,\tau_{1}\right]$ . We have $\left|Q\left(M^{0,\tau_{1}}\right)\right|\leq4cSR^{2}$
(\eqref{eq:jjdpoa}). Then 
\begin{equation}
\left|tQ\left(M^{0,\tau_{1}}\right)\left(s_{0},x+c\left(s_{0}-t\right),y\right)\right|\leq\tau_{1}4cSR^{2}
\end{equation}
and (\eqref{eq:kqiiq-1-1})
\begin{equation}
-\tau_{1}4cSR^{2}+\min_{\left[a_{1},b_{1}\right]\times\left[a_{2},b_{2}\right]}N^{0}_{1}\leq\mathcal{T}^{0,\tau_{1},A}_{1}\left(M^{0,\tau_{1}}\right)\left(t,x,y\right).
\end{equation}
Recall (for example \cite{arxiv 4V} proof of Prop. 4.1) that the
bounds of each integral in the formulas \eqref{eq:olole-1-3}-\eqref{eq:olole-1-1-1-2-2}
are in the interval $\left[\tau,\tau'\right]$ . We then have (\eqref{eq:kqiiq-1-1-1})
\begin{equation}
-\tau_{1}4cSR^{2}+\min_{\left[0,T\right]\times\left[a_{2},b_{2}\right]}N^{-}_{1}\leq\mathcal{T}^{0,\tau_{1},B}_{1}\left(M^{0,\tau_{1}}\right)\left(t,x,y\right).
\end{equation}
We establish that similar relations hold for $\mathcal{T}^{0,\tau_{1},A}_{i}$
and $\mathcal{T}^{0,\tau_{1},B}_{i}$ ($i=2,3,4$). Then we see that,
if 
\begin{equation}
T<\frac{1}{4cSR^{2}}\min_{1\leq i\leq4}\left\{ \right.\min_{\left[a_{1},b_{1}\right]\times\left[a_{2},b_{2}\right]}N^{0}_{i},\min_{\left[0,T\right]\times\left[a_{2},b_{2}\right]}N^{-}_{1},
\min_{\left[0,T\right]\times\left[a_{1},b_{1}\right]}N^{-}_{2},\min_{\left[0,T\right]\times\left[a_{1},b_{1}\right]}N^{+}_{3},\min_{\left[0,T\right]\times\left[a_{2},b_{2}\right]}N^{+}_{4}\left.\right\} \label{mpod}
\end{equation}
then $\mathcal{T}^{0,T}\left(M^{0,T}\right)=M^{0,T}>0.$ If \eqref{mpod}
does not hold, then taking 
\begin{equation}
\tau_{1}=\frac{1}{8cSR^{2}}\min_{1\leq i\leq4}\left\{ \right.\min_{\left[a_{1},b_{1}\right]\times\left[a_{2},b_{2}\right]}N^{0}_{i},\min_{\left[0,T\right]\times\left[a_{2},b_{2}\right]}N^{-}_{1},
\min_{\left[0,T\right]\times\left[a_{1},b_{1}\right]}N^{-}_{2},\min_{\left[0,T\right]\times\left[a_{1},b_{1}\right]}N^{+}_{3},\min_{\left[0,T\right]\times\left[a_{2},b_{2}\right]}N^{+}_{4}\left.\right\} \label{mpod-1}
\end{equation}
yields $\mathcal{T}^{0,\tau_{1}}\left(M^{0,\tau_{1}}\right)=M^{0,\tau_{1}}>0.$

Let $\tau_{2}$such that $\tau_{1}<\tau_{2}\leq T.$ Consider the
problem $\Sigma_{\tau_{1},\tau_{2}}$ where the initial data are $N^{\tau_{1}}\equiv M^{0,\tau_{1}}\left(\tau_{1},\cdot,\cdot\right).$
We have as previously, for $i=1,2,3,4$
\begin{equation}
-\left(\tau_{2}-\tau_{1}\right)4cSR^{2}+\min_{\left[a_{1},b_{1}\right]\times\left[a_{2},b_{2}\right]}N^{\tau_{1}}_{i}\leq\mathcal{T}^{\tau_{1},\tau_{2},A}_{i}\left(M^{\tau_{1},\tau_{2}}\right)\left(t,x,y\right)
\end{equation}
and similar relations for $\mathcal{T}^{\tau_{1},\tau_{2},B}_{i}.$
Where from if 
\begin{equation}
T-\tau_{1}<\frac{1}{4cSR^{2}}\min_{1\leq i\leq4}\left\{ \right.\min_{\left[a_{1},b_{1}\right]\times\left[a_{2},b_{2}\right]}N^{\tau_{1}}_{i},\min_{\left[0,T\right]\times\left[a_{2},b_{2}\right]}N^{-}_{1},
\min_{\left[0,T\right]\times\left[a_{1},b_{1}\right]}N^{-}_{2},\min_{\left[0,T\right]\times\left[a_{1},b_{1}\right]}N^{+}_{3},\min_{\left[0,T\right]\times\left[a_{2},b_{2}\right]}N^{+}_{4}\left.\right\} \label{mpod-2}
\end{equation}
then $M^{\tau_{1},T}>0$. If \eqref{mpod-2} does not hold, then for
$\tau_{2}$such that 
\begin{equation}
\tau_{2}-\tau_{1}=\frac{1}{8cSR^{2}}\min_{1\leq i\leq4}\left\{ \right.\min_{\left[a_{1},b_{1}\right]\times\left[a_{2},b_{2}\right]}N^{\tau_{1}}_{i},\min_{\left[0,T\right]\times\left[a_{2},b_{2}\right]}N^{-}_{1},
\min_{\left[0,T\right]\times\left[a_{1},b_{1}\right]}N^{-}_{2},\min_{\left[0,T\right]\times\left[a_{1},b_{1}\right]}N^{+}_{3},\min_{\left[0,T\right]\times\left[a_{2},b_{2}\right]}N^{+}_{4}\left.\right\} \label{mpod-2-1}
\end{equation}
we have $M^{\tau_{1},\tau_{2}}>0$.

Suppose by induction that we have $0<\tau_{1}<\cdots<\tau_{n}\leq T$
such that $M^{\tau_{k-1},\tau_{k}}>0$ for $k=2,\cdots,n.$ By similar
arguments as previously, take $N^{\tau_{n}}\equiv M^{\tau_{n-1},\tau_{n}}\left(\tau_{n},\cdot,\cdot\right).$
If 
\begin{equation}
T-\tau_{n}<\frac{1}{4cSR^{2}}\min_{1\leq i\leq4}\left\{ \right.\min_{\left[a_{1},b_{1}\right]\times\left[a_{2},b_{2}\right]}N^{\tau_{n}}_{i},\min_{\left[0,T\right]\times\left[a_{2},b_{2}\right]}N^{-}_{1},
\min_{\left[0,T\right]\times\left[a_{1},b_{1}\right]}N^{-}_{2},\min_{\left[0,T\right]\times\left[a_{1},b_{1}\right]}N^{+}_{3},\min_{\left[0,T\right]\times\left[a_{2},b_{2}\right]}N^{+}_{4}\left.\right\} \label{mpod-2-2}
\end{equation}
then $M^{\tau_{n},T}>0$. If \eqref{mpod-2-2} does not hold, then
for $\tau_{n+1}$ such that 
\begin{equation}
\tau_{n+1}-\tau_{n}=\frac{1}{8cSR^{2}}\min_{1\leq i\leq4}\left\{ \right.\min_{\left[a_{1},b_{1}\right]\times\left[a_{2},b_{2}\right]}N^{\tau_{n}}_{i},\min_{\left[0,T\right]\times\left[a_{2},b_{2}\right]}N^{-}_{1},
\min_{\left[0,T\right]\times\left[a_{1},b_{1}\right]}N^{-}_{2},\min_{\left[0,T\right]\times\left[a_{1},b_{1}\right]}N^{+}_{3},\min_{\left[0,T\right]\times\left[a_{2},b_{2}\right]}N^{+}_{4}\left.\right\} \label{mpod-2-1-1}
\end{equation}
we have $M^{\tau_{n},\tau_{n+1}}>0$.

Now suppose that this procedure of extension of the positiveness breaks
down for an interval $\left[0,T'\right[$ such that $T'<T$. Then
$M^{0,T'}>0$ on $\left[0,T'\right[$ and by continuity of $M,$ $M^{0,T'}>0$
on $\left[0,T'\right].$ Then if 
\begin{equation}
T-T'<\frac{1}{4cSR^{2}}\min_{1\leq i\leq4}\left\{ \right.\min_{\left[a_{1},b_{1}\right]\times\left[a_{2},b_{2}\right]}M^{0,T'}_{i}\left(T',\cdot,\cdot\right),\min_{\left[0,T\right]\times\left[a_{2},b_{2}\right]}N^{-}_{1},
\min_{\left[0,T\right]\times\left[a_{1},b_{1}\right]}N^{-}_{2},\min_{\left[0,T\right]\times\left[a_{1},b_{1}\right]}N^{+}_{3},\min_{\left[0,T\right]\times\left[a_{2},b_{2}\right]}N^{+}_{4}\left.\right\} \label{mpod-2-2-1}
\end{equation}
then we obtain $M^{T',T}>0;$otherwise taking $T''\in\left]T',T\right]$
such that 
\begin{equation}
T''-T'=\frac{1}{8cSR^{2}}\min_{1\leq i\leq4}\left\{ \right.\min_{\left[a_{1},b_{1}\right]\times\left[a_{2},b_{2}\right]}M^{0,T'}_{i}\left(T',\cdot,\cdot\right),\min_{\left[0,T\right]\times\left[a_{2},b_{2}\right]}N^{-}_{1},
\min_{\left[0,T\right]\times\left[a_{1},b_{1}\right]}N^{-}_{2},\min_{\left[0,T\right]\times\left[a_{1},b_{1}\right]}N^{+}_{3},\min_{\left[0,T\right]\times\left[a_{2},b_{2}\right]}N^{+}_{4}\left.\right\} \label{mpod-2-1-1-1}
\end{equation}
yields $M^{T',T''}>0.$ Thus, the extension of positiveness reaches
$\left[0,T\right].$ Hence we have $M^{0,T}>0,$ for all $T$ such
that $M$ is defined on $\left[0,T\right];$ we conclude that $M>0.$
\end{proof}

\section{Existence }\label{sec:Existence}
\begin{rem}
\label{rem:In--we}In \cite{arxiv 4V} we consider the subspace $E$
of $C\left(\begin{array}{r}
\mathscr{P}\end{array};\R\right)$ consisting of functions $u$ that are continuous on $\mathscr{P}$
such that $\dfrac{\partial u}{\partial t},\dfrac{\partial u}{\partial x},\dfrac{\partial u}{\partial y}$
are defined in $\mathring{\mathscr{P}}$ except possibly in the planes
with respective equations $x-c\left(t-\tau\right)=a_{1},y-c\left(t-\tau\right)=a_{2},y+c\left(t-\tau\right)=b_{2}$
and $x+c\left(t-\tau\right)=b_{1}$, and are continuous and bounded.
We know that $\forall M\in E^{4}$ such that $\left(\dfrac{\partial}{\partial t},\dfrac{\partial}{\partial x},\dfrac{\partial}{\partial y}\right)\left|M\right|$
exist, it holds that $\mathcal{T}^{\sigma}\left(M\right)\in E^{4}.$
Particularly this holds for all $M\in E^{4},$ such that $M\geq0.$ 
\end{rem}

Using the Definition \ref{nnjnjn}, we consider for the present
problem, the sets 
\begin{equation}
\mathscr{M}_{R,R'}\equiv\left\{ \right.M\in E^{4}/\left\Vert M\right\Vert <R\text{ and }\max\left\{ \left\Vert \dfrac{\partial M}{\partial t}\right\Vert ,\left\Vert \dfrac{\partial M}{\partial x}\right\Vert ,\left\Vert \dfrac{\partial M}{\partial y}\right\Vert \right\} <R'\left.\right\} \label{oizuue}
\end{equation}
and $\mathscr{M}^{+}_{R,R'}$ which are non-empty convex and relatively
compact in $\left(C\left(\begin{array}{r}
\mathscr{P}\end{array};\R\right)^{4},\left\Vert \cdot\right\Vert \right)$ according to Proposition \ref{prop::odlp-1}. 
\begin{prop}
\label{prop:Let--be}Let $R_{0}>0$ be given and consider $R>0$ such
that $R<R_{0}.$ Consider also $\left[\tau,\tau'\right]$ such that
$\tau'-\tau\leq1.$ Then we have the following estimations: For every
$R'>0,$and for every $M\in\mathscr{M}^{+}_{R,R'},$

\begin{equation}
\left\Vert \mathcal{T}^{\sigma}\left(M\right)\right\Vert <\left(\tau'-\tau\right)\left(4\sigma+4cS\right)R^{2}+q\label{lodo}
\end{equation}
and 
\begin{multline}
\max\left\{ \left\Vert \dfrac{\partial}{\partial t}\mathcal{T}^{\sigma}\left(M\right)\right\Vert ,\left\Vert \dfrac{\partial}{\partial x}\mathcal{T}^{\sigma}\left(M\right)\right\Vert ,\left\Vert \dfrac{\partial}{\partial y}\mathcal{T}^{\sigma}\left(M\right)\right\Vert \right\} <
\left(8+\dfrac{4}{c}\right)\left(\sigma+cS\right)R^{2}_{0}+4\sigma\left(4\sigma+4cS\right)R^{3}_{0}+\left(8+\dfrac{4}{c}\right)\sigma R_{0}q+q'+\\
\left[4\sigma\left(1+c\right)\left(4\sigma+4cS\right)R^{2}_{0}+\left(1+c\right)\left(8\sigma+8cS\right)R_{0}+4\sigma\left(1+c\right)q\right]\left(\tau'-\tau\right)R'.\label{kdid}
\end{multline}
\end{prop}

\begin{proof}
It follows from \eqref{eq:kiid} and \eqref{eq:kiid} that 
\begin{equation}
\left\Vert Q^{\sigma}_{i}\left(M\right)\right\Vert _{\infty}\leq\left(4\sigma+4cS\right)\left\Vert M\right\Vert ^{2},\label{kdiod}
\end{equation}
\begin{equation}
\left\Vert \left(\dfrac{\partial}{\partial t},\dfrac{\partial}{\partial x},\dfrac{\partial}{\partial y}\right)Q^{\sigma}_{i}\left(M\right)\right\Vert \leq\left(8\sigma+8cS\right)\left\Vert M\right\Vert \cdot\underbrace{\max\left\{ \left\Vert \dfrac{\partial M}{\partial t}\right\Vert ,\left\Vert \dfrac{\partial M}{\partial x}\right\Vert ,\left\Vert \dfrac{\partial M}{\partial y}\right\Vert \right\} }_{=\max_{1\leq i\leq4}\left\Vert \nabla_{t,x,y}M_{i}\right\Vert }
\end{equation}
\begin{equation}
\left\Vert \rho\left(M\right)\right\Vert _{\infty}\leq4\left\Vert M\right\Vert \label{ksidi}
\end{equation}
 and 
\begin{equation}
\left\Vert \left(\dfrac{\partial}{\partial t},\dfrac{\partial}{\partial x},\dfrac{\partial}{\partial y}\right)\rho\left(M\right)\right\Vert \leq4\left\Vert M\right\Vert .\label{idood}
\end{equation}

Rewrite \eqref{eq:olole-1} as 
\begin{multline}
\mathcal{T}^{A,\sigma}_{1}\left(M\right)\left(t,x,y\right)=
{\displaystyle \int^{t}_{\tau}}e^{-\sigma\int^{t}_{s}\rho\left(\left|M\right|\right)\left(r,x+c\left(r-t\right),y\right)dr}\cdot Q^{\sigma}_{1}\left(\left|M\right|\right)\left(s,x+c\left(s-t\right),y\right)ds+\\
\overline{N^{\tau}_{1}}\left(t,x,y\right)\cdot e^{-\sigma{\textstyle \int^{t}_{\tau}}\rho\left(\left|M\right|\right)\left(s,x+c\left(s-t\right),y\right)ds}.\label{eq:olole-1-4}
\end{multline}
Recall that the lower bound in every integral is less or equal to
the upper bound in the formulas \eqref{aoalal}-\eqref{eq:olole-1-1-1-2}.
Using \eqref{kdiod}, \eqref{eq:olole-1-4} yields 

\begin{align}
\left\Vert \mathcal{T}^{A,\sigma}_{1}\left(M\right)\right\Vert _{\infty} & <\left(\tau'-\tau\right)\left(4\sigma+4cS\right)R^{2}+q.\label{eq:ksoloa-2}
\end{align}
Compute $\dfrac{\partial}{\partial t}\mathcal{T}^{A,\sigma}_{1}\left(M\right)$
from \eqref{eq:olole-1-4} and estimate $\left\Vert \dfrac{\partial}{\partial t}\mathcal{T}^{A,\sigma}_{1}\left(M\right)\right\Vert _{\infty}$
using \eqref{kdiod}-\eqref{idood}:
\begin{multline}
\left\Vert \dfrac{\partial}{\partial t}\mathcal{T}^{A,\sigma}_{1}\left(M\right)\right\Vert _{\infty}\leq\left(4\sigma+4cS\right)\left\Vert M\right\Vert ^{2}+
\left(\tau'-\tau\right)\left\{ \right.\sigma\left(\right.4\left\Vert M\right\Vert +\left(\tau'-\tau\right)c\cdot4\max_{1\leq i\leq4}\left\Vert \nabla_{t,x,y}M_{i}\right\Vert \left.\right)\cdot
\left(4\sigma+4cS\right)\left\Vert M\right\Vert ^{2}+\\c\left(8\sigma+8cS\right)\left\Vert M\right\Vert \max_{1\leq i\leq4}\left\Vert \nabla_{t,x,y}M_{i}\right\Vert \left.\right\} +
\sigma\left(\right.4\left\Vert M\right\Vert +\left(\tau'-\tau\right)c\cdot4\max_{1\leq i\leq4}\left\Vert \nabla_{t,x,y}M_{i}\right\Vert \left.\right)\cdot q+q'.\label{eq:oookkz-2-5}
\end{multline}
Use $\left\Vert M\right\Vert <R<R_{0},$ $\max_{1\leq i\leq4}\left\Vert \nabla_{t,x,y}M_{i}\right\Vert <R'$
and $\tau'-\tau\leq1;$ it follows 
\begin{multline}
\left\Vert \dfrac{\partial}{\partial t}\mathcal{T}^{A,\sigma}_{1}\left(M\right)\right\Vert _{\infty}<
\left(4\sigma+4cS\right)R^{2}_{0}+4\sigma\left(4\sigma+4cS\right)R^{3}_{0}+4\sigma R_{0}q+q'+\\
\left[4\sigma c\left(4\sigma+4cS\right)R^{2}_{0}+c\left(8\sigma+8cS\right)R_{0}+4\sigma cq\right]\left(\tau'-\tau\right)R'.\label{eq:oookkz-2-5-1}
\end{multline}
Similar calculations and estimations for $\left(\dfrac{\partial}{\partial x},\dfrac{\partial}{\partial y}\right)\mathcal{T}^{A,\sigma}_{1}\left(M\right)$
yield 
\begin{equation}
\left\Vert \left(\dfrac{\partial}{\partial x},\dfrac{\partial}{\partial y}\right)\mathcal{T}^{A,\sigma}_{1}\left(M\right)\right\Vert <
q'+\left[4\sigma\left(4\sigma+4cS\right)R^{2}_{0}+\left(8\sigma+8cS\right)R_{0}+4\sigma q\right]\left(\tau'-\tau\right)R'.\label{eq:oookkz-2-5-1-1}
\end{equation}
Following the similar method for \ref{eq:olole-1-1} it follows 
\begin{equation}
\left\Vert \mathcal{T}^{B,\sigma}_{1}\left(M\right)\right\Vert _{\infty}<\left(\tau'-\tau\right)\left(4\sigma+4cS\right)R^{2}+q\label{eq:omplo-1}
\end{equation}
\begin{multline}
\left\Vert \dfrac{\partial}{\partial t}\mathcal{T}^{B,\sigma}_{1}\left(M\right)\right\Vert _{\infty}<\left(8\sigma+8cS\right)R^{2}_{0}+4\sigma\left(4\sigma+4cS\right)R^{3}_{0}+8\sigma R_{0}q+q'+\\
\left[4\sigma c\left(4\sigma+4cS\right)R^{2}_{0}+c\left(8\sigma+8cS\right)R_{0}+4\sigma cq\right]\left(\tau'-\tau\right)R'\label{eq:oookkz-2-3-1}
\end{multline}
\begin{equation}
\left\Vert \dfrac{\partial}{\partial x}\mathcal{T}^{B,\sigma}_{1}\left(M\right)\right\Vert _{\infty}<\dfrac{1}{c}\left(4\sigma+4cS\right)R^{2}_{0}+\dfrac{4}{c}\sigma R_{0}q+q'+
\left[4\sigma\left(4\sigma+4cS\right)R^{2}_{0}+\left(8\sigma+8cS\right)R_{0}+4\sigma q\right]\left(\tau'-\tau\right)R'\label{eq:oookkz-2-1-2-1}
\end{equation}
\begin{equation}
\left\Vert \dfrac{\partial}{\partial y}\mathcal{T}^{B,\sigma}_{1}\left(M\right)\right\Vert _{\infty}<
q'+\left[4\sigma\left(4\sigma+4cS\right)R^{2}_{0}+\left(8\sigma+8cS\right)R_{0}+4\sigma q\right]\left(\tau'-\tau\right)R'\label{eq:oookkz-2-1-1-1-1}
\end{equation}
Now it follows from \eqref{eq:ksoloa-2}, \eqref{eq:oookkz-2-5-1},
\eqref{eq:oookkz-2-5-1-1}, \eqref{eq:oookkz-2-3-1}, \eqref{eq:oookkz-2-1-2-1}
and \eqref{eq:oookkz-2-1-1-1-1} that 
\begin{align}
\left\Vert \mathcal{T}^{\sigma}_{1}\left(M\right)\right\Vert _{\infty} & <\left(\tau'-\tau\right)\left(4\sigma+4cS\right)R^{2}+q\label{eq:ksoloa-2-1}
\end{align}
and 
\begin{multline}
\left\Vert \nabla_{t,x,y}\mathcal{T}^{\sigma}_{1}\left(M\right)\right\Vert <
\left(8+\dfrac{4}{c}\right)\left(\sigma+cS\right)R^{2}_{0}+4\sigma\left(4\sigma+4cS\right)R^{3}_{0}+\left(8+\dfrac{4}{c}\right)\sigma R_{0}q+q'+\\
\left[4\sigma\left(1+c\right)\left(4\sigma+4cS\right)R^{2}_{0}+\left(1+c\right)\left(8\sigma+8cS\right)R_{0}+4\sigma\left(1+c\right)q\right]\left(\tau'-\tau\right)R'.\label{jddud}
\end{multline}
The formulas \eqref{aoalal-2}-\eqref{eq:olole-1-1-1-2-2} are similar,
and the previous method of estimation yields the same bound for $\left\Vert \mathcal{T}^{\sigma}_{i}\left(M\right)\right\Vert _{\infty}$
and $\left\Vert \nabla_{t,x,y}\mathcal{T}^{\sigma}_{i}\left(M\right)\right\Vert $
as in \eqref{eq:ksoloa-2-1} and \eqref{jddud} when $i=2,3,4$. It
follows that 
\begin{align}
\left\Vert \mathcal{T}^{\sigma}\left(M\right)\right\Vert  & <\left(\tau'-\tau\right)\left(4\sigma+4cS\right)R^{2}+q\label{eq:ksoloa-2-1-1}
\end{align}
and 
\begin{multline}
\max_{1\leq i\leq4}\left\Vert \nabla_{t,x,y}\mathcal{T}^{\sigma}_{i}\left(M\right)\right\Vert <
\left(8+\dfrac{4}{c}\right)\left(\sigma+cS\right)R^{2}_{0}+4\sigma\left(4\sigma+4cS\right)R^{3}_{0}+\left(8+\dfrac{4}{c}\right)\sigma R_{0}q+q'+\\
\left[4\sigma\left(1+c\right)\left(4\sigma+4cS\right)R^{2}_{0}+\left(1+c\right)\left(8\sigma+8cS\right)R_{0}+4\sigma\left(1+c\right)q\right]\left(\tau'-\tau\right)R'.\label{jddud-1}
\end{multline}
\end{proof}

\begin{proof}[\textbf{ Proof of theorem \ref{thm:dhgrt}}]
As mentioned in the Preliminaries, the uniqueness of solutions to $\Sigma_{\tau,\tau'}$
and the continuity of the operator $\mathcal{T}^{\sigma}$ are already
established in \cite{arxiv 4V}. \\
$\tau'-\tau\leq\min\left\{ 1,f\left(q\right),g\left(q\right)\right\} $
implies that $\tau'-\tau\leq\dfrac{1}{16\left(\sigma+cS\right)q}.$
\\
Take $R=\dfrac{1-\sqrt{1-16\left(\tau'-\tau\right)\left(\sigma+cS\right)q}}{\left(8\sigma+8cS\right)\left(\tau'-\tau\right)}$.
It follows that 
\begin{equation}
R<R_{0}\iff1-\left(8\sigma+8cS\right)\left(\tau'-\tau\right)R_{0}<\sqrt{1-16\left(\tau'-\tau\right)\left(\sigma+cS\right)q}.\label{loco}
\end{equation}
From \eqref{ldod} $\tau'-\tau\leq f\left(q\right)$ implies that
$\tau'-\tau<\dfrac{1}{\left(8\sigma+8cS\right)R_{0}},$ that is\\ $1-\left(8\sigma+8cS\right)\left(\tau'-\tau\right)R_{0}>0.$
Hence \eqref{loco} rewrites successively 
\begin{align}
R<R_{0}\iff & 16\left(\tau'-\tau\right)\left(\sigma+cS\right)\left(q-R_{0}\right)+\left(8\sigma+8cS\right)^{2}\left(\tau'-\tau\right)^{2}R^{2}_{0}<0\label{loco-1}\\
\iff & q-R_{0}+\left(4\sigma+4cS\right)\left(\tau'-\tau\right)R^{2}_{0}<0.
\end{align}
As $q>R_{0}$ it holds that 
\begin{equation}
R<R_{0}\iff\tau'-\tau<\dfrac{R_{0}-q}{\left(4\sigma+4cS\right)R^{2}_{0}}.
\end{equation}
From the hypotheses and \eqref{loodie}, $\tau'-\tau\leq g\left(q\right)<\dfrac{R_{0}-q}{\left(4\sigma+4cS\right)R^{2}_{0}}.$
Hence it holds that $R<R_{0}.$ The hypotheses also imply that $\tau'-\tau\leq1.$
Therefore $R$ and $\tau,\tau'$ satisfy the hypotheses in the Proposition
\ref{prop:Let--be}. Consequently \eqref{lodo} and \eqref{kdid}
hold. Therefore for every $R'>0,$ $\mathcal{T}^{\sigma}\left(\mathscr{M}_{R,R'}\right)$
is a subset of a certain $\mathscr{M}_{R_{1},R_{1}'}$ where from
$\mathcal{T}^{\sigma}$ is compact on $\mathscr{M}_{R,R'}.$ \\
Now $R=\dfrac{1-\sqrt{1-16\left(\tau'-\tau\right)\left(\sigma+cS\right)q}}{\left(8\sigma+8cS\right)\left(\tau'-\tau\right)}$
implies that $R$ satisfies the inequality\\ $\left(\tau'-\tau\right)\left(4\sigma+4cS\right)R^{2}+q\leq R.$
From \eqref{lodo}, it follows that 
\begin{equation}
M\in\mathscr{M}^{+}_{R,R'}\implies\left\Vert \mathcal{T}^{\sigma}\left(M\right)\right\Vert <R.\label{lsoiki}
\end{equation}
On another hand, 
\begin{multline}
\left(8+\dfrac{4}{c}\right)\left(\sigma+cS\right)R^{2}_{0}+4\sigma\left(4\sigma+4cS\right)R^{3}_{0}+\left(8+\dfrac{4}{c}\right)\sigma R_{0}q+q'+\\
\left[4\sigma\left(1+c\right)\left(4\sigma+4cS\right)R^{2}_{0}+\left(1+c\right)\left(8\sigma+8cS\right)R_{0}+4\sigma\left(1+c\right)q\right]\left(\tau'-\tau\right)R'\leq R'\label{idido}
\end{multline}
is equivalent to 
\begin{multline}
R'\left\{ 1-\left[4\sigma\left(1+c\right)\left(4\sigma+4cS\right)R^{2}_{0}+\left(1+c\right)\left(8\sigma+8cS\right)R_{0}+4\sigma\left(1+c\right)q\right]\left(\tau'-\tau\right)\right\} \geq\\
\left(8+\dfrac{4}{c}\right)\left(\sigma+cS\right)R^{2}_{0}+4\sigma\left(4\sigma+4cS\right)R^{3}_{0}+\left(8+\dfrac{4}{c}\right)\sigma R_{0}q+q'.\label{oloo}
\end{multline}
From \eqref{ldod} $\tau'-\tau\leq f\left(q\right)$ implies that
\begin{equation}
\tau'-\tau<\dfrac{1}{4\sigma\left(1+c\right)\left(4\sigma+4cS\right)R^{2}_{0}+\left(1+c\right)\left(8\sigma+8cS\right)R_{0}+4\sigma\left(1+c\right)q}.
\end{equation}
Therefore \eqref{idido} is equivalent to 
\begin{equation}
R'\geq\dfrac{\left(8+\dfrac{4}{c}\right)\left(\sigma+cS\right)R^{2}_{0}+4\sigma\left(4\sigma+4cS\right)R^{3}_{0}+\left(8+\dfrac{4}{c}\right)\sigma R_{0}q+q'}{1-\left(\tau'-\tau\right)\left[4\sigma\left(1+c\right)\left(4\sigma+4cS\right)R^{2}_{0}+\left(1+c\right)\left(8\sigma+8cS\right)R_{0}+4\sigma\left(1+c\right)q\right]}.\label{oleiu}
\end{equation}
Take 
\begin{equation}
R'=\dfrac{\left(8+\dfrac{4}{c}\right)\left(\sigma+cS\right)R^{2}_{0}+4\sigma\left(4\sigma+4cS\right)R^{3}_{0}+\left(8+\dfrac{4}{c}\right)\sigma R_{0}q+q'}{1-\left(\tau'-\tau\right)\left[4\sigma\left(1+c\right)\left(4\sigma+4cS\right)R^{2}_{0}+\left(1+c\right)\left(8\sigma+8cS\right)R_{0}+4\sigma\left(1+c\right)q\right]}.\label{oddd}
\end{equation}
 \eqref{idido} is satisfied. It follows from \eqref{kdid} that 
\begin{equation}
M\in\mathscr{M}^{+}_{R,R'}\implies\max\left\{ \left\Vert \dfrac{\partial}{\partial t}\mathcal{T}^{\sigma}\left(M\right)\right\Vert ,\left\Vert \dfrac{\partial}{\partial x}\mathcal{T}^{\sigma}\left(M\right)\right\Vert ,\left\Vert \dfrac{\partial}{\partial y}\mathcal{T}^{\sigma}\left(M\right)\right\Vert \right\} <R'.\label{jduyh}
\end{equation}
From Remark \ref{rem:In--we}, $M\in\mathscr{M}^{+}_{R,R'}\implies\mathcal{T}^{\sigma}\left(M\right)\in E^{4}.$
Therefore \eqref{lsoiki} and \eqref{jduyh} imply that $\mathcal{T}^{\sigma}\left(\mathscr{M}^{+}_{R,R'}\right)\subset\mathscr{M}^{+}_{R,R'}$
for $\sigma$ sufficiently large as $\mathcal{T}^{\sigma}$ is non-negative.
\\
Finally, as $\mathscr{M}^{+}_{R,R'}$ is a non-empty convex subset
of $C\left(\begin{array}{r}
\mathscr{P}\end{array};\R\right)^{4},$we deduce from the Schauder fixed point theorem that $\mathcal{T}^{\sigma}$
possesses a fixed point $N\in\mathscr{M}^{+}_{R,R'};$from the Proposition
\eqref{prop:pqmm} $N$ is a non-negative solution to $\Sigma_{\tau,\tau'}.$
Moreover, $N\in\mathscr{M}^{+}_{R,R'}\implies$
\begin{equation}
\left\Vert N\right\Vert <R<R_{0}\label{iido-1}
\end{equation}
and 
\begin{equation}
\max\left\{ \left\Vert \dfrac{\partial N}{\partial t}\right\Vert ,\left\Vert \dfrac{\partial N}{\partial x}\right\Vert ,\left\Vert \dfrac{\partial N}{\partial y}\right\Vert \right\} <R'.\label{ossisi}
\end{equation}
\end{proof}

\begin{proof}[\textbf{Proof of theorem \ref{thm:The-problem-}}]
Put $\tau_{0}=0,$ 
\begin{equation}
q_{0}\equiv{\displaystyle \max_{1\leq i\leq4}}\Biggl\{\left\Vert \overline{N^{0}_{i}}\right\Vert _{\infty},\left\Vert \overline{N^{-}_{1}}\right\Vert _{\infty},\left\Vert \overline{N^{-}_{2}}\right\Vert _{\infty},\left\Vert \overline{N^{+}_{3}}\right\Vert _{\infty},\left\Vert \overline{N^{+}_{4}}\right\Vert _{\infty}\Biggr\}\label{sskiu}
\end{equation}
 and 
\begin{equation}
q'_{0}\equiv\max_{1\leq i\leq4}\left\{ \left\Vert \nabla_{t,x,y}\overline{N^{0}_{i}}\right\Vert ,\left\Vert \nabla_{t,x,y}\overline{N^{-}_{1}}\right\Vert ,\left\Vert \nabla_{t,x,y}\overline{N^{-}_{2}}\right\Vert ,\left\Vert \nabla_{t,x,y}\overline{N^{+}_{3}}\right\Vert ,\left\Vert \nabla_{t,x,y}\overline{N^{+}_{4}}\right\Vert \right\} .\label{oloo-1-1}
\end{equation}
and consider 
\begin{equation}
R_{0}>q_{0}.\label{sssqq}
\end{equation}
Put $\tau_{1}-\tau_{0}\equiv S_{0}=\min\left\{ 1,f\left(q_{0}\right),g\left(q_{0}\right)\right\} .$
Then according to the Theorem \ref{thm:dhgrt}, the problem $\Sigma_{\tau_{0},\tau_{1}}$
admits an unique solution $N^{\tau_{0},\tau_{1}}\geq0$ defined on
$\left[\tau_{0},\tau_{1}\right]\times\left[a_{1},b_{1}\right]\times\left[a_{2},b_{2}\right]$
and $N^{\tau_{0},\tau_{1}}$ satisfies 
\begin{equation}
\left\Vert N^{\tau_{0},\tau_{1}}\right\Vert <R_{0}\label{iido-2}
\end{equation}
and 
\begin{equation}
\max\left\{ \left\Vert \dfrac{\partial N^{\tau_{0},\tau_{1}}}{\partial t}\right\Vert ,\left\Vert \dfrac{\partial N^{\tau_{0},\tau_{1}}}{\partial x}\right\Vert ,\left\Vert \dfrac{\partial N^{\tau_{0},\tau_{1}}}{\partial y}\right\Vert \right\} <\dfrac{V\left(R_{0}\right)+\left(8+\frac{4}{c}\right)\sigma R_{0}q_{0}+q'_{0}}{1-S_{0}\left[W\left(R_{0}\right)+4\sigma\left(1+c\right)q_{0}\right]}\label{oizu-1}
\end{equation}
where 
\begin{equation}
V\left(R_{0}\right)\equiv\left(8+\frac{4}{c}\right)\left(\sigma+cS\right)R^{2}_{0}+4\sigma\left(4\sigma+4cS\right)R^{3}_{0}\label{jduud}
\end{equation}
and 
\begin{equation}
W\left(R_{0}\right)\equiv4\sigma\left(1+c\right)\left(4\sigma+4cS\right)R^{2}_{0}+\left(1+c\right)\left(8\sigma+8cS\right)R_{0}.\label{osisisu}
\end{equation}
 Consider the problem $\Sigma_{\tau_{1},\tau_{2}}$ where the initial
data $N^{\tau_{1}}_{i}$ are defined by 
\begin{equation}
N^{\tau_{1}}_{i}\left(\cdot,\cdot\right)\equiv N^{\tau_{0},\tau_{1}}_{i}\left(\tau_{1},\cdot,\cdot\right).\label{ldood}
\end{equation}
 As $N^{\tau_{0},\tau_{1}}$ is solution to $\Sigma_{\tau_{0},\tau_{1}},$
the boundary conditions \eqref{eq:losso-1}-\eqref{eq:ikioi-1-1}
for $\Sigma_{\tau_{0},\tau_{1}}$ when $t=\tau_{1},$ writes 
\begin{align*}
N^{\tau_{0},\tau_{1}}_{1}\left(\tau_{1},a_{1},y\right)= & N^{-}_{1}\left(\tau_{1},y\right)\\
N^{\tau_{0},\tau_{1}}_{2}\left(\tau_{1},x,a_{2}\right)= & N^{-}_{2}\left(\tau_{1},x\right)\\
N^{\tau_{0},\tau_{1}}_{3}\left(\tau_{1},x,b_{2}\right)= & N^{+}_{3}\left(\tau_{1},x\right)\\
N^{\tau_{0},\tau_{1}}_{4}\left(\tau_{1},b_{1},y\right)= & N^{+}_{4}\left(\tau_{1},y\right)
\end{align*}
hence the compatibility conditions \eqref{eq:kqiiq-1}-\eqref{eq:looqp-1-1}
for $\Sigma_{\tau_{1},\tau_{2}}$ are satisfied. Put
\begin{equation}
q_{1}\equiv{\displaystyle \max_{1\leq i\leq4}}\Biggl\{\left\Vert \overline{N^{\tau_{1}}_{i}}\right\Vert _{\infty},\left\Vert \overline{N^{-}_{1}}\right\Vert _{\infty},\left\Vert \overline{N^{-}_{2}}\right\Vert _{\infty},\left\Vert \overline{N^{+}_{3}}\right\Vert _{\infty},\left\Vert \overline{N^{+}_{4}}\right\Vert _{\infty}\Biggr\}\label{lodidu}
\end{equation}
and 
\begin{equation}
q'_{1}\equiv\max_{1\leq i\leq4}\left\{ \left\Vert \nabla_{t,x,y}\overline{N^{\tau_{1}}_{i}}\right\Vert ,\left\Vert \nabla_{t,x,y}\overline{N^{-}_{1}}\right\Vert ,\left\Vert \nabla_{t,x,y}\overline{N^{-}_{2}}\right\Vert ,\left\Vert \nabla_{t,x,y}\overline{N^{+}_{3}}\right\Vert ,\left\Vert \nabla_{t,x,y}\overline{N^{+}_{4}}\right\Vert \right\} .\label{oloo-1-1-1}
\end{equation}
It holds with $\gamma\equiv1+c+\dfrac{1}{c}$ that 
\begin{equation}
\left\Vert \nabla_{t,x,y}\overline{N^{\tau_{1}}_{i}}\right\Vert \leq\gamma\left\Vert \nabla_{t,x,y}N^{\tau_{1}}_{i}\right\Vert \label{oeutyy}
\end{equation}
and hence from \eqref{oizu-1} that 
\begin{equation}
q'_{1}\leq\gamma\dfrac{V\left(R_{0}\right)+\left(8+\frac{4}{c}\right)\sigma R_{0}q_{0}+q'_{0}}{1-S_{0}\left[W\left(R_{0}\right)+4\sigma\left(1+c\right)q_{0}\right]}.
\end{equation}
From \eqref{eq:kqiiq-1-1}-\eqref{eq:looqp-1-1-1}, \eqref{ldood},
\eqref{iido-2} it follows that 
\[
\left\Vert \overline{N^{\tau_{1}}_{i}}\right\Vert _{\infty}\leq\left\Vert N^{\tau_{1}}_{i}\right\Vert _{\infty}<R_{0}.
\]
As $R_{0}>q_{0},$ it holds (from \eqref{sskiu}) that 
\[
R_{0}>{\displaystyle \max_{1\leq i\leq4}}\Biggl\{\left\Vert \overline{N^{-}_{1}}\right\Vert _{\infty},\left\Vert \overline{N^{-}_{2}}\right\Vert _{\infty},\left\Vert \overline{N^{+}_{3}}\right\Vert _{\infty},\left\Vert \overline{N^{+}_{4}}\right\Vert _{\infty}\Biggr\}
\]
yielding $R_{0}>q_{1}.$\\
Put $\tau_{2}-\tau_{1}\equiv S_{1}=\min\left\{ 1,f\left(q_{1}\right),g\left(q_{1}\right)\right\} .$
Then the problem $\Sigma_{\tau_{1},\tau_{2}}$ admits an unique solution
$N^{\tau_{1},\tau_{2}}\geq0$ defined on $\left[\tau_{1},\tau_{2}\right]\times\left[a_{1},b_{1}\right]\times\left[a_{2},b_{2}\right]$.
The solution $N^{\tau_{1},\tau_{2}}$ satisfies 
\begin{equation}
\left\Vert N^{\tau_{1},\tau_{2}}\right\Vert <R_{0}\label{iido-2-1}
\end{equation}
and 
\begin{equation}
\max\left\{ \left\Vert \dfrac{\partial N^{\tau_{0},\tau_{1}}}{\partial t}\right\Vert ,\left\Vert \dfrac{\partial N^{\tau_{0},\tau_{1}}}{\partial x}\right\Vert ,\left\Vert \dfrac{\partial N^{\tau_{0},\tau_{1}}}{\partial y}\right\Vert \right\} <\dfrac{V\left(R_{0}\right)+\left(8+\dfrac{4}{c}\right)\sigma R_{0}q_{1}+q'_{1}}{1-S_{1}\left[W\left(R_{0}\right)+4\sigma\left(1+c\right)q_{1}\right]}.\label{oizu-1-1}
\end{equation}
 Assume by induction that there are $\tau_{0},\tau_{1},\cdots,\tau_{n}>0$
such that $\forall\nu=0,1,\cdots,n-1,$

{*} the problem $\Sigma_{\tau_{\nu},\tau_{\nu+1}}$ posed on $\left[\tau_{\nu},\tau_{\nu+1}\right]\times\left[a_{1},b_{1}\right]\times\left[a_{2},b_{2}\right]$where
$N^{\tau_{\nu}}_{i}\left(\cdot,\cdot\right)\equiv N^{\tau_{\nu-1},\tau_{\nu}}_{i}\left(\tau_{\nu},\cdot,\cdot\right)$
are the initial data, admits an unique solution $N^{\tau_{\nu},\tau_{\nu+1}}\geq0$ 

{*} $\tau_{\nu+1}-\tau_{\nu}\equiv S_{\nu}=\min\left\{ 1,f\left(q_{\nu}\right),g\left(q_{\nu}\right)\right\} $
with 
\begin{equation}
q_{\nu}\equiv{\displaystyle \max_{1\leq i\leq4}}\Biggl\{\left\Vert \overline{N^{\tau_{\nu}}_{i}}\right\Vert _{\infty},\left\Vert \overline{N^{-}_{1}}\right\Vert _{\infty},\left\Vert \overline{N^{-}_{2}}\right\Vert _{\infty},\left\Vert \overline{N^{+}_{3}}\right\Vert _{\infty},\left\Vert \overline{N^{+}_{4}}\right\Vert _{\infty}\Biggr\}\label{lodidu-1}
\end{equation}

{*} and that $N^{\tau_{\nu},\tau_{\nu+1}}$ satisfies 
\begin{equation}
\left\Vert N^{\tau_{\nu},\tau_{\nu+1}}\right\Vert <R_{0}\label{lsso}
\end{equation}
 and 
\begin{equation}
\max\left\{ \left\Vert \dfrac{\partial N^{\tau_{\nu},\tau_{\nu+1}}}{\partial t}\right\Vert ,\left\Vert \dfrac{\partial N^{\tau_{\nu},\tau_{\nu+1}}}{\partial x}\right\Vert ,\left\Vert \dfrac{\partial N^{\tau_{\nu},\tau_{\nu+1}}}{\partial y}\right\Vert \right\} <\dfrac{V\left(R_{0}\right)+\left(8+\frac{4}{c}\right)\sigma R_{0}q_{\nu}+q'_{\nu}}{1-S_{\nu}\left[W\left(R_{0}\right)+4\sigma\left(1+c\right)q_{\nu}\right]}.\label{ldodo}
\end{equation}
Put $N^{\tau_{n}}_{i}\left(\cdot,\cdot\right)\equiv N^{\tau_{n-1},\tau_{n}}_{i}\left(\tau_{n},\cdot,\cdot\right)$
and $q_{n}={\displaystyle \max_{1\leq i\leq4}}\Biggl\{\left\Vert \overline{N^{\tau_{n}}_{i}}\right\Vert _{\infty},\left\Vert \overline{N^{-}_{1}}\right\Vert _{\infty},\left\Vert \overline{N^{-}_{2}}\right\Vert _{\infty},\left\Vert \overline{N^{+}_{3}}\right\Vert _{\infty},\left\Vert \overline{N^{+}_{4}}\right\Vert _{\infty}\Biggr\}.$
The compatibility conditions \eqref{eq:kqiiq-1}-\eqref{eq:looqp-1-1}
for $\Sigma_{\tau_{n},\tau_{n+1}}$ are satisfied as the boundary
conditions \eqref{eq:losso-1}-\eqref{eq:ikioi-1-1} for $\Sigma_{\tau_{n-1},\tau_{n}}$
when $t=\tau_{n}.$ From \eqref{lsso}, it follows that $q_{n+1}>R_{0}.$
Thus taking $\tau_{n+1}-\tau_{n}\equiv S_{n}=\min\left\{ 1,f\left(q_{n}\right),g\left(q_{n}\right)\right\} $
and applying Theorem \eqref{thm:dhgrt} yield the existence of an
unique solution $N^{\tau_{n},\tau_{n+1}}\geq0$ on $\left[S_{n},S_{n+1}\right]\times\left[a_{1},b_{1}\right]\times\left[a_{2},b_{2}\right]$
to the problem $\Sigma_{\tau_{n},\tau_{n+1}},$ and $N^{\tau_{n},\tau_{n+1}}$
satisfies 
\begin{equation}
\left\Vert N^{\tau_{n},\tau_{n+1}}\right\Vert <R_{0}\label{lsso-1}
\end{equation}
 and 
\begin{equation}
\max\left\{ \left\Vert \dfrac{\partial N^{\tau_{n},\tau_{n+1}}}{\partial t}\right\Vert ,\left\Vert \dfrac{\partial N^{\tau_{n},\tau_{n+1}}}{\partial x}\right\Vert ,\left\Vert \dfrac{\partial N^{\tau_{n},\tau_{n+1}}}{\partial y}\right\Vert \right\} <\dfrac{V\left(R_{0}\right)+\left(8+\frac{4}{c}\right)\sigma R_{0}q_{n}+q'_{n}}{1-S_{n}\left[W\left(R_{0}\right)+4\sigma\left(1+c\right)q_{n}\right]}.\label{ldodo-1}
\end{equation}
Therefore a sequence $\left(N^{\tau_{n},\tau_{n+1}}\right)_{n\geq0}$
of non-negative continuous uniformly bounded solutions to the system
\eqref{eq:koiqmn} is constructed. For every $n,$ $N^{\tau_{n},\tau_{n+1}}$
is defined on $\left[\tau_{n},\tau_{n+1}\right]\times\left[a_{1},b_{1}\right]\times\left[a_{2},b_{2}\right]$
and satisfies the boundary conditions \eqref{eq:losso}-\eqref{eq:ikioi-1};
$N^{\tau_{0},\tau_{1}}$ satisfies the initial conditions \eqref{eq:lsos}
whilst for $n\geq1,$ $N^{\tau_{n},\tau_{n+1}}$ satisfies the initial
conditions $N^{\tau_{n},\tau_{n+1}}\left(\tau_{n},x,y\right)=N^{\tau_{n-1},\tau_{n}}\left(\tau_{n},x,y\right),$
$\forall\left(x,y\right)\in\left[a_{1},b_{1}\right]\times\left[a_{2},b_{2}\right].$
\\
By construction for every $n\geq0,$ 
\begin{equation}
\tau_{n+1}-\tau_{n}\equiv S_{n}=\min\left\{ 1,f\left(q_{n}\right),g\left(q_{n}\right)\right\} >0;\label{losod}
\end{equation}
 hence the sequence $\left(\tau_{n}\right)_{n\geq0}$ of real positive
numbers either converges or diverges to $\infty.$ On the one hand,
from \eqref{ldod}, it follows that $\left(f\left(q_{n}\right)\right)$
converges to zero if and only if $\left(q_{n}\right)$ diverges to
$\infty.$ But 
\begin{equation}
q_{n}\equiv{\displaystyle \max_{1\leq i\leq4}}\Biggl\{\left\Vert \overline{N^{\tau_{n}}_{i}}\right\Vert _{\infty},\left\Vert \overline{N^{-}_{1}}\right\Vert _{\infty},\left\Vert \overline{N^{-}_{2}}\right\Vert _{\infty},\left\Vert \overline{N^{+}_{3}}\right\Vert _{\infty},\left\Vert \overline{N^{+}_{4}}\right\Vert _{\infty}\Biggr\}<R_{0}
\end{equation}
therefore it is impossible that $\left(f\left(q_{n}\right)\right)$
converges to zero.\\
 On the other hand, from \eqref{loodie}, $\left(g\left(q_{n}\right)\right)$
converges to zero if and only if $\left(q_{n}\right)$ converges to
$R_{0}.$ But for every $n\geq0,$ $q_{n}<R_{0}$ and from \eqref{sssqq},
$R_{0}$ is taken arbitrarily such that $R_{0}>q_{0}.$ Therefore
it holds that 
\begin{equation}
\forall n\geq0,\,q_{n}\leq q_{0}.
\end{equation}
It follows that $\left(q_{n}\right)$ cannot converge to $R_{0},$
hence $\left(g\left(q_{n}\right)\right)$ cannot converge to zero.
It follows from \eqref{losod} that $\left(S_{n}\right)$ cannot converge
to zero. \\
From \eqref{losod}, ${\displaystyle \sum^{n}_{k=0}}S_{k}={\displaystyle \sum^{n}_{k=0}\left(\tau_{k+1}-\tau_{k}\right)=\tau_{n+1}-\tau_{0}=\tau_{n+1},}$
therefore $\left(\tau_{n}\right)_{n\geq0}$ diverges to $\infty.$

As a direct conclusion, the function $N$ whose restriction to $\left[0,\tau_{1}\right[$
is $N^{0,\tau_{0}}$ and to $\left[\tau_{n},\tau_{n+1}\right[$ is
$N^{\tau_{n},\tau_{n+1}},$$\left(\forall n\geq0\right)$ is non-negative,
defined and continuous on $\left[0,+\infty\right[\times\left[a_{1},b_{1}\right]\times\left[a_{2},b_{2}\right]$
and is the unique solution to the problem $\Sigma.$ \\
Now by construction, for all $R_{0}>q_{0},$ $\left\Vert N^{\tau_{n},\tau_{n+1}}\right\Vert <R_{0}$
$,\forall n\geq0.$ Therefore for all $R_{0}>q_{0},$ $\left\Vert N\right\Vert <R_{0}.$
It follows that 
\begin{equation}
\left\Vert N\right\Vert \leq q_{0}.
\end{equation}
 But from \eqref{sskiu}, it holds (see \eqref{eq:kqiiq-1-1}-\eqref{eq:looqp-1-1-1-1})
that 
\begin{equation}
q_{0}\leq{\displaystyle \max_{1\leq i\leq4}}\left\{ \right.\left\Vert N^{0}_{i}\right\Vert _{\infty},\left\Vert N^{-}_{1}\right\Vert _{\infty},\left\Vert N^{-}_{2}\right\Vert _{\infty},\left\Vert N^{+}_{3}\right\Vert _{\infty},\left\Vert N^{+}_{4}\right\Vert _{\infty}\left.\right\} .
\end{equation}
We conclude 
\begin{equation}
\left\Vert N\right\Vert \leq{\displaystyle \max_{1\leq i\leq4}}\left\{ \right.\left\Vert N^{0}_{i}\right\Vert _{\infty},\left\Vert N^{-}_{1}\right\Vert _{\infty},\left\Vert N^{-}_{2}\right\Vert _{\infty},\left\Vert N^{+}_{3}\right\Vert _{\infty},\left\Vert N^{+}_{4}\right\Vert _{\infty}\left.\right\} .
\end{equation}
\end{proof}

\section*{Conclusion }

In this paper, we have successfully extended the global existence
and uniqueness theory for the multidimensional initial-boundary value
problem associated with the four-velocity Broadwell model. By leveraging
and adapting our previously established fixed-point framework, we
have demonstrated that the restriction to sufficiently small initial
data can be completely overcome. Our approach establishes that classical
solutions, along with their first-order partial derivatives, exist
and remain well-defined globally in time for arbitrary (large) initial
and boundary data that are bounded in $C^{1}$, with bounded first-order
derivatives. While this work provides an advancement in the existence
theory of classical solutions of non-stationary multidimensional discrete
kinetic equations, it also opens up intriguing paths for future research.
Specifically, although the solutions and their derivatives are proved
to exist for all $t\in[0,+\infty[$, their long-term asymptotic behavior
remains unresolved. Determining whether these derivatives remain uniformly
bounded or grow indefinitely as $t\to+\infty$ constitutes a challenging
open question. Furthermore, a natural and promising extension of this
work will be to apply our fixed-point methodology to general discrete
kinetic equations involving binary or multiple collisions. Investigating
both this structural generalization and the asymptotic stability of
the derivatives will be the subject of forthcoming investigations.

\end{document}